\documentclass{article}
\usepackage{graphicx}

\usepackage[fleqn]{amsmath}
\usepackage{amsthm}
\usepackage{amsmath}
\usepackage{amssymb}
\usepackage{amsbsy}
\usepackage{amsfonts}
\usepackage{mathrsfs}
\usepackage[all]{xy}
\usepackage{amstext}
\usepackage{amscd}
\usepackage[dvips]{epsfig}
\usepackage{psfrag}
\usepackage{enumerate}
\usepackage{flafter}
\allowdisplaybreaks

\theoremstyle{plain}
\newtheorem{thm}{Theorem}[section]
\newtheorem{prop}[thm]{Proposition}
\newtheorem{cor}[thm]{Corollary}
\newtheorem{lem}[thm]{Lemma}
\theoremstyle{definition}
\newtheorem{exa}[thm]{Example}

\newtheorem{rem}[thm]{Remark}
\newtheorem{defn}[thm]{Definition}

\def\det{\mathop{\mathrm{det}}\nolimits}

\def\Ker{\mathop{\mathrm{Ker}}\nolimits}

\def\F{\mathop{\mathbb{F}}\nolimits}

\newcommand{\lra}{\longrightarrow}
\newcommand{\ra}{\rightarrow}
\newcommand{\Q}{{\Bbb Q}}
\newcommand{\R}{{\Bbb R}}
\newcommand{\Z}{{\Bbb Z}}
\newcommand{\N}{{\Bbb N}}

\newcommand{\pc}[2]{\mbox{$\begin{array}{c}
\includegraphics[scale=#2]{#1.eps}
\end{array}$}}

\begin{document}

\large
\begin{center}{\bf\Large Twisted Alexander invariants of knot group representations II; computation and duality}.
\end{center}
\vskip 1.5pc
\begin{center}{Takefumi Nosaka\footnote{
E-mail address: {\tt nosaka@math.titech.ac.jp}
}}\end{center}
\vskip 1pc
\begin{abstract}\baselineskip=12pt \noindent
Given a homomorphism from a link group to a group,
we introduce a $K_1$-class, which is a generalization of the 1-variable Alexander polynomial.
We compare the $K_1$-class with
$K_1$-classes in \cite{Nos} and with Reidemeister torsions. 
As a corollary, we show a relation to Reidemeister torsions of finite cyclic covering spaces, and show reciprocity in some senses.
\end{abstract}
\begin{center}
\normalsize
\baselineskip=17pt
{\bf Keywords} \\
\ \ \ knot, Alexander polynomial, $K_1$-group, Novikov ring \ \
\end{center}

\large
\baselineskip=16pt
\section{Introduction}\label{IntroS}

Let $ L \subset S^3$ be a knot in the 3-sphere.
The Alexander polynomial of $L$ has been studied in many ways, and
its applications and beautiful properties are discovered,
including, e.g., relations to Reidemeister torsions, Fox derivatives, and
applications to cyclic coverings and slice knots.
As a generalization, after twisted Alexander polynomials are introduced \cite{Wada,Lin},
similar applications and properties are discovered (however, such results frequently require some assumptions); see \cite{FV,FKK} and references therein. 

In the previous paper \cite{Nos}, which is inspired by \cite{Mil,Tur}, 
given a homomorphism from a link group to a group,the author suggested elements of
a $K_1$-group, which are a generalization of the twisted Alexander polynomials.
This paper gives several approaches to the $K_1$-class in another way.
In Sections \ref{defS2} and \ref{invrelS3}, we give two other definitions as $K_1$-classes from the viewpoint of the Fox derivatives or Reidemeister torsion over $K_1$.
Here, an advantage of the definitions is applicable to not only knots but also links; see Definition \ref{def11}.
Furthermore, in the knot case,
we show the equivalence (up to some ambiguity) of the three definitions (Theorem \ref{cor1155}).
As a corollary of the equivalence, 
we compute the $K_1$-classes of some 2-bridge knots; see Section \ref{invrelS772}.

In Sections \ref{def732}--\ref{ExaSS77}, we will address some applications.
First, under some conditions, we give a relation to (commutative) Reidemeister torsion of the $m$-fold cyclic covering space of $S^3 \setminus L$; see Section \ref{def732}.
Next, we will see that the conditions are suitable to group homomorphisms $f: \pi_1(S^3 \setminus L) \ra \Z \rtimes G$, where $G$ is a finite group.
In this situation,
we show (Theorem \ref{pro366}) a duality theorem of the $K_1$-class, as in reciprocity of the (twisted) Alexander polynomials \cite{FKK,Mil,Tur}.
In application, we give an estimate of sliceness of knots, which is a slight generalization of the works of Herald-Kirk-Livingston \cite{HKL, KL}.
In fact, we find that the conditions in Section \ref{ExaSS77} are applicable to the Casson-Gordan theory \cite{CG}.

Finally, under some conditions, we give a relation to the circle valued Morse theory (see Appendix \ref{paji}), and a reduction to the higher order Alexander polynomial \cite{C,Har}; see Appendix \ref{invrelS4}. These discussions rely on the works of \cite{Paji,Fri2}.



\

\noindent
{\textbf{Conventional notation.} For a group $G$ and a commutative ring $A$, we denote the group ring by $A [G]$, and the abelianization by $G_{\rm ab}$.
Every non-commutative ring $R$ has always 1, and is assumed to satisfy that $R^r $ and $R^s$ with $r \neq s$ are not isomorphic as $R$-modules.
We mean by $R^{\times}$ the multiplicative group consisting of units in $R$.

\subsection*{Acknowledgments}
The author expresses his gratitude to Takahiro Kitayama, Ryoto Tange for valuable comments.


\section{Definition of the Alexander $K_1$-class}\label{defS2}

We give the definition as a generalization of the Alexander polynomial; see Definition \ref{def11}. 

As a similar setting to \cite{Nos}, we set up algebraic terminologies and need an assumption.
Let $\mathcal{A}$ be a ring, which is possibly non-commutative, and take a ring isomorphism $\kappa : \mathcal{A} \ra \mathcal{A} $. 
Then, we have the completed skew Laurent-polynomial ring $\mathcal{A}_{\kappa}( \!( \tau) \!) $.
Namely, $\mathcal{A}_{\kappa}( \!( \tau) \!) $ is the set of formal power series $ \sum_{i= - N }^{\infty}a_i \tau^i $ where $a_i \in \mathcal{A}$ and $ \tau^n a = \kappa^n (a) \tau^n $; in other words, $\mathcal{A}_{\kappa}( \!( \tau) \!) $ is equal to $\mathcal{A}_{\kappa}[ \![ \tau] \!][\tau^{-1}] $, which is also called the Novikov ring in \cite{Fri,PR}.
Furthermore, let us fix a group $G$ with presentation $ \langle x_1, \dots, x_m | r_1, \dots, r_{m-1} \rangle$ of deficiency 1, 
and consider the following assumption throughout this paper.

\vskip 0.61pc

\noindent
\underline{\bf Assumption $(\dagger)$}
Let $\mathcal{A} $ be a ring and $\kappa :\mathcal{A} \ra \mathcal{A}$ be as above. 
Suppose a ring homomorphism $\rho: \Z[ G ] \ra \mathcal{A}_{\kappa}( \!( \tau) \!) $ such
that, for any $i \leq m$, there is $ w_i \in \mathcal{A}_{\kappa}( \!( \tau) \!)^{\times} $ such taht $\rho (x_i) = w_i^{-1} \tau w_i $.

\begin{exa}[{A special case of Novikov completion}]\label{exppp0}
Set a link $L$ in the 3-sphere $S^3$. Choose a link diagram $D$, and let $m$ be the number of the arcs on $D$.
Then, the Wirtinger presentation from $D$ gives such a presentation $\pi_1(S^3 \setminus L) \cong \langle x_1, \dots, x_m | r_1, \dots, r_{m-1} \rangle$, possibly $r_j=1$ (see, e.g., \cite{Lic}, \cite[\S 5]{Wada}).

Given a group homomorphism $h: \pi_1(S^3 \setminus L) \ra H \rtimes \Z$ such that $H$ is a group, $h (x_i) =(g_i, 1) $
for some $g_i \in G$.
Let $\mathcal{A} $ be the group ring $B[H]$ over a commutative ring $B$.
If we replace $ h(x_i )$ by $h(x_i ) \tau $ and define $\kappa(a):= (0,1) (a,0) (0,-1)$ for $a \in H$
we have $\mathcal{A}_{\kappa}( \!( \tau) \!) $. This $h$ canonically gives rise to $ \rho : \Z[ \pi_1(S^3 \setminus L)] \ra \mathcal{A}_{\kappa}( \!( \tau) \!) $ satisfying $(\dagger).$

In general, we obtain such an $h$
from any group $G$ and any group homomorphism $f: \pi_1(S^3 \setminus L) \ra G $,
as follows.
Let $\mathrm{Ab}: \pi_1(S^3 \setminus L ) \ra \Z^{\sharp L}$ be an abelianization,
and $\mu: \Z ^{\sharp L} \ra \Z$ be the multiplication such that $ \mu \circ \mathrm{Ab}(x_i) =1$ for any $i$. Notice the isomorphism $\pi_1(S^3 \setminus L) \cong \Ker ( \mu \circ \mathrm{Ab}) \rtimes \Z $.
Let $ H \subset G$ be the restricted image $f( \Ker( \mu \circ \mathrm{Ab}) ) $, and
let $\Z = \{ \tau^n \}_{n \in \Z} $ act on $H$ by $ g \cdot \tau^n:= f(x_1)^{n} g f( x_1)^{-n} $.
By the action, we have the semi-direct product $H \rtimes \Z $, and
can define a homomorphism $$ h: \pi_1(S^3 \setminus L) = \Ker ( \mu \circ \mathrm{Ab}) \rtimes \Z \ra H \rtimes \Z \ \ \ \mathrm{by} \ \ \ h( g, n )= (f(g) ,n), $$
which satisfy the condition in the previous paragraph.
\end{exa}

Let us review up $K_1$-groups. 
For a ring $R$ with unit, let $GL_n(R)$ be the general linear group over $R$ of size $n$.
Since $ GL_n(R)$ diagonally injects into $GL_{n+1}(R) $, we have the colimit $GL(R) = \lim GL_n(R) $.
The $K_1$-group, $K_1(R)$, is defined to be the abelianization $ GL(R)_{\rm ab}$.
By abuse of notation, we often regard elements of $GL_n(R)$ as those of $K_1(R)$.
When $ R=\mathcal{A}_{\kappa}( \!( \tau) \!)$ as above,
$ -1$, $ \tau$, and $ \rho (g)$ are represented by the invertible $(1\times 1)$-matrices. Let
$\pm \tau$ denote the subgroup of $K_1( \mathcal{A}_{\kappa}( \!( \tau) \!))$ generated by
$-1$ and $\tau$, which is isomorphic to either $ \Z$ or $\Z \times \Z/2$.
Similarly, let $\pm \rho( G) $ be the subgroup of $K_1( \mathcal{A}_{\kappa}( \!( \tau) \!))$ generated by $-1$ and the image $\rho (G)$.
We later use the quotient groups $ K_1( \mathcal{A}_{\kappa}( \!( \tau) \!))/ \pm \tau$ and $ K_1( \mathcal{A}_{\kappa}( \!( \tau) \!))/ \pm \rho( G) $.



Furthermore, we study the Fox derivative and a Jocobi matrix.
Let $ \mathcal{F}$ be the free group with a basis, $x_1,\dots, x_{m}$. 
We define a $\Z$-linear map $\frac{\partial \ }{\partial x_i} : \Z [ \mathcal{F}] \ra \Z[ \mathcal{F}]$ by the following identities:
$$\frac{\partial x_j }{\partial x_i}= \delta_{ij}, \ \ \ \ \frac{\partial (hk) }{\partial x_i}= \frac{\partial h }{\partial x_i}+ h \frac{\partial k}{\partial x_i} \ \ \ \ \ \ \ (h,k \in F) . $$
Consider the $(m-1)\times m$ matrix, $A_{\rho}$, over $\mathcal{A}_{\kappa}( \!( \tau) \!) $ whose ($i,j$)-th entry
is $ \rho ( \frac{\partial r_i}{ \partial x_j}) \in \mathcal{A}_{\kappa}( \!( \tau) \!) $.
For $ 1 \leq k \leq m$, let us denote by $A_{\rho, k}$ by the $(m-1)\times (m-1)$-matrix obtained from $A_{\rho}$ by removing the $k$-th column.

\begin{defn}\label{def11}
Choose $k \in \mathbb{N}$ as above.
Suppose that $A_{\rho, k}$ is an invertible matrix.
We define {\it the $K_1$-Alexander class (with respect to $\rho$)} to be the $K_1$-class 
$$ [ A_{\rho, k}] \in K_1( \mathcal{A}_{\kappa}( \!( \tau) \!))/ \pm \tau . $$
On the other hand, if $A_{\rho, k}$ is not invertible, we define the $K_1$-class to be zero.
\end{defn}

\begin{thm}\label{cor11}
Let $G= \pi_1(S^3 \setminus L)$ be a link group with a Wirtinger presentation, as in Example \ref{exppp0}.
Then, the $K_1$-class $\Delta_{\rho }^{K_1}$ depends only on the homomorphism $\rho: \Z[\pi_1(S^3 \setminus L)] \ra \mathcal{A}_{\kappa}( \!( \tau) \!) $.
\end{thm}

We will give the proof in \S \ref{lplp}.
Although we suppose invertibility of $A_{F,W}$, we give a criterion for the invertibility:

\begin{prop}\label{prop11} Suppose $\sharp L=1$, i.e., $L$ is a knot.
Then, $L$ is fibered if and only if $A_{\rho, k} $ is invertible for any homomorphism $\rho$ satisfying ($\dagger$).
\end{prop}

We conclude this section by explaining that the twisted Alexander polynomial of \cite{Wada} can be formulated from our $K_1$-Alexander class:
\begin{exa}\label{exa112}
Let $R$ be a commutative ring.
The paper \cite[\S 5]{Wada} defines a polynomial from a representation $\rho^{\rm pre}: \pi_1(S^3 \setminus L) \ra SL_n(R)$ as follows.
For the formulation, let $\mathcal{A} $ be the matrix ring $\mathrm{Mat}(n \times n,R)$, 
and let $\kappa$ be the identity $\mathrm{id}_{\mathcal{A} }$.
Formally, let $\tau$ be $\rho^{\rm pre} (\mathfrak{m})$ as a commutative indeterminate. 
Since there is a ring homomorphism
$ \Z[ SL_n(R)] \ra \mathrm{Mat}(n \times n,R) $ which
sends $\sum_{g} a_g g$ to $\sum_{g} a_g g$, the $ \rho^{\rm pre} $ gives rise to a ring homomorphism $\rho:\Z[\pi_1(S^3\setminus L) ] \ra \mathcal{A}_{\kappa}( \!( \tau) \!)$.
Notice that the determinant $\mathrm{Mat}(n \times n,R_{\kappa} ( \!( \tau) \!)) \ra R_{\kappa} ( \!( \tau) \!)$ induces a homomorphism
$$\mathrm{det}: K_1(\mathcal{A}_{\kappa}( \!( \tau) \!) )/ \pm \tau \lra R_{\kappa} ( \!( \tau) \!)^{\times}/ \pm \tau . $$
Then, if $\# L =1$, checking \cite[Corollary 5]{Wada} carefully, we verify by construction that the determinant $\det (\Delta_{\rho}^{K_1} ) \in R_{\kappa } ( \!( \tau) \!)^{\times} / \pm \tau $ exactly coincides with the twisted Alexander polynomial of \cite{Wada}.
In other words, our $K_1$-class $\Delta_{\rho}^{K_1}$ is a lift of the twisted Alexander polynomial.
\end{exa}

\subsection{Computation for some 2-bridge knots}\label{invrelS772}

We will compute the $K_1$-classes with respect to every 2-bridge knots of Seifert genus 1. 

\

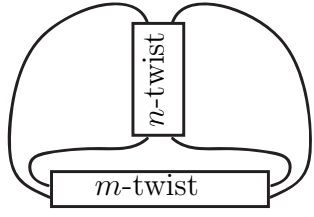
\begin{figure}[htpb]
\begin{center}
\begin{picture}(210,60)
\put(61,32){\pc{twist.knot2}{0.37104}}
\put(117,28){\large \rotatebox{90}{$n $-twist}}
\put(99,2){\large $m $-twist}
\end{picture}
\end{center}
\caption{\label{bbbbb} The 2-bridge knot of genus 1.}
\end{figure}

For non-zero integers $m$ and $n$, let $K(m, n)$
be the 2-bridge knot $S(4mn+1, 2m)$ in Schubert's form; see Figure \ref{bbbbb}. 
We may suppose $m>0$.
As is known, every 2-bridge knot of genus 1 is represented as one of $K(m, n)$.
Thanks to Proposition 2.1 in \cite{HT}, the knot group
$\pi_1(S^3 \setminus K(m, n))$ has the presentation
$$ \langle x,y | w^n x= y w^n\rangle , \ \ \ \ \mathrm{where} \ \ w= (xy^{-1})^m ( x^{-1}y)^{m}, $$
where $ x$ and $y$ are conjugate to a meridian.
Moreover, as seen in the proof of the proposition, this presentation is
strongly Tietze equivalent to a Wirtinger presentation.
\begin{prop}\label{torus2}Let $m>0.$ Consider the following element of $ \mathcal{A}_{\kappa}( \!( \tau) \!) $:
\[ Z_{m,n}:= \begin{cases} \rho (w^n) +(1- \rho (y) ) \frac{1-\rho(w^n)}{1-\rho(w)}\bigl( 1 - \rho(w x^{-1} ) \bigr) \frac{1-\rho(x y^{-1})^m}{1-\rho(xy^{-1})} & (n>0), \\
\rho (w^n) +( \rho (y)-1 )\rho(w^{-1}) \frac{1-\rho(w^n)}{1-\rho(w)}\bigl( 1 - \rho(w x^{-1} ) \bigr) \frac{1-\rho(x y^{-1})^m}{1-\rho(xy^{-1})}& (n<0 ).
\end{cases} \]
The matrix $A_{F,W}$ is invertible if and only if $Z_{m,n}$ is an invertible element.
If so, the $K_1$-Alexander invariant is given by
$$ \Phi_{\rho,k }= Z_{m,n} \in K_1( \mathcal{A}_{\kappa}( \!( \tau) \!))/ \pm \tau .$$
\end{prop}
\begin{proof} Let $r=w^n x w^{-n}y^{-1} $.
It is sufficient to show $Z_{m,n}= \rho ( \frac{\partial r}{ \partial x})$.
Notice that 
\[ \frac{\partial r}{ \partial x} = \frac{\partial w^n}{ \partial x} + w^n \frac{\partial \ \ }{ \partial x}( x w^{-n} y^{-1})= \frac{\partial w^n}{ \partial x} + w^n (1+ x\frac{\partial \ \ }{ \partial x}( w^{-n})) \]
\[ = \begin{cases}
(1+w+ \cdots +w^{n-1}) \frac{\partial w}{ \partial x} +w^n + y w^n(1+w^{-1}+ \cdots +w^{1-n }) \frac{\partial w^{-1}}{ \partial x} & (n>0), \\
(1+w^{-1}+ \cdots +w^{n+1}) \frac{\partial w}{ \partial x} +w^n + y w^n(1+w^{-1}+ \cdots +w^{-n-1 }) \frac{\partial w^{-1}}{ \partial x }& (n<0 ),
\end{cases} \]
\[ = \begin{cases}
w^n + (1-y)(1+w+ \cdots +w^{n-1 }) \frac{\partial w}{ \partial x} & (n>0) ,\\
w^n + (y-1)(w^{-1}+ \cdots +w^{n }) \frac{\partial w}{ \partial x} & (n<0 ).
\end{cases}\]
Here, the last equality is due to $\frac{\partial w^{-1}}{ \partial x} = -w^{-1} \frac{\partial w}{ \partial x}$.
We should observe
\[\frac{\partial w}{ \partial x} = (1+ xy^{-1}+ \cdots + (xy^{-1})^{m-1})- (xy^{-1})^m (1+ x^{-1}y+ \cdots + (x^{-1}y)^{m-1})x^{-1} \]
\[=(1 -(xy^{-1})^m y^{-1} (y x^{-1})^m )(1+ xy^{-1}+ \cdots + (xy^{-1})^{m-1} ) \]
\[=(1 -w x^{-1}) (1 -(xy^{-1})^{m} )/(1- xy^{-1}) . \]
Therefore, we have $Z_{m,n}= \rho ( \frac{\partial r}{ \partial x})$ as required.
\end{proof}
The above computation is done as an element of $\mathcal{A}_{\kappa}( \!( \tau) \!))$; however,
there are many examples
such that $\Phi_{\rho,k }$ the $K_1$-Alexander invariant $\Phi_{\rho,k } $ can not be represented by any $1 \times 1$-matrix.
e.g, $L$ is a non-fibered pretzel knot.
Indeed, even concerning the classical Alexander module,
the minimal number of sizes of the matrix $ \Phi_{\rho,k }$ is estimated by Nakanishi index.

\subsection{Proof of Theorem \ref{cor11}}\label{lplp}

For the proof, we review the strongly Tietze transformations \cite{Wada}.
For a finite presentable group $G= \langle g_1, \dots, g_m | r_1, \dots, r_n \rangle$ and a
word $w$ of $g_1,\dots, g_m$, the transformation of the following types are called the {\it strongly Tietze transformations}:

(Ia) To replace one of the relators $r_i$ by its inverse $r_i^{-1}.$

(Ib)
To replace one of the relators $r_i$ by its conjugate $w r_i^{-1} w^{-1}.$

(Ic) To replace one of the relators $r_i$ by $r_i r_j $ for any $j \neq i$.

(II) To add a new generator $y$ and a new relator $ y w^{-1}.$ In other words, the resulting presentation is given by
$ \langle g_1, \dots, g_m ,y | r_1, \dots, r_n ,yw^{-1}\rangle $.

\

Two presentations of $G$ are said to be {\it strongly Tietze equivalent}, if they are
related by a finite sequence of the operations of the above types and their inverse operations.
\begin{proof}[Proof of Theorem \ref{cor11}]
It is shown \cite[Lemma 6 in \S 5]{Wada} that any two Wirtinger presentations for a given link $L$ are strongly Tietze equivalent.
Hence, for the proof of Theorem \ref{cor11}, it is enough to show the following proposition, which claims an invariance with respect to strongly Tietze equivalence:
\end{proof}
\begin{prop}[{cf. \cite[Theorem 4.5]{Kit}}]\label{thm11}
The $K_1$-class $[ \Phi_{\rho, k}]$ in $K_1( \mathcal{A}_{\kappa}( \!( \tau) \!))/ \pm \tau $ does not depend on the choice of $k$.
Moreover, if we change another presentation of $ G$ which is strongly Tietz equivalent to the above presentation, the associated $K_1$-class $[ \Phi_{\rho, k}]$ in $K_1( \mathcal{A}_{\kappa}( \!( \tau) \!))/ \pm \tau$ is invariant.
\end{prop}
\begin{proof}
The proof is essentially based on the proofs of \cite[Lemma 2 in \S 3]{Wada} and \cite[Theorem 4.5]{Kit}.
Following the proofs, we will check $(i)$ the independence of the choice of $k$, and $(ii)$ the invariance with respect to each transformations (Ia)(Ib)(Ic)(II).

To show $(i)$, choose $k'$. Let $e_{j}(a) $ be the diagonal $(m-1)^2$-matrix
whose $(j,j) $-th entry is $a$ and $ (k,k)$-th entry is $1$ where $j \neq k$.
Then, according to the fundamental formula
$$ \sum_{j=1}^m \frac{\partial r_i}{\partial x_j} (x_j -1) =r_i -1$$
(see \cite[(2.3)]{Fox}), we have
\[ A_{\rho, k'} e_k ( \rho(x_k) -1 ) = \Bigl(\dots, \rho \Bigl( \frac{\partial r_i}{\partial x_k} \Bigr) \rho(x_k -1) ,\dots \Bigr)= \Bigl(\dots, - \sum_{j \neq k }\rho \Bigl( \frac{\partial r_i}{\partial x_j} \Bigr) \rho(x_j -1) ,\dots \Bigr) \]
\[= \Bigl(\dots, -\rho \Bigl( \frac{\partial r_i}{\partial x_{k'}} \Bigr) \rho(x_{k'} -1) ,\dots \Bigr) = (-1)^{k-k'} A_{\rho, k} e_{k'} ( \rho(x_{k'}) -1 ) \in \mathrm{Mat}( m\times m; \mathcal{A}_{\kappa}( \!( \tau) \!) ). \]
Notice that $e_k ( \rho(x_k) -1 ) $ is invertible, since
$ 1 - \rho (x_k )$ is invertible because of
$ \bigl( 1 - \rho (x_k )\bigr)(1 + \sum_{j=1}^{\infty } w_k^{-1 } \tau^j w_k )=1$.
Therefore, the invertibility of $A_{\rho,k}$ implies that $ A_{\rho, k'} $ is also invertible.
Moreover, since $ \rho(x_{k}) -1= \rho(w_k^{-1}) (\tau -1) \rho(w_k) = \tau -1 $ in $K_1$ by Assumption $(\dagger)$, we obtain $A_{\rho, k'} = A_{\rho, k} $ in the quotient $K_1( \mathcal{A}_{\kappa}( \!( \tau) \!))/ \pm \tau$.

Next, concerning $(ii)$, we consider strongly Tietz transformations.
Notice
\[\frac{\partial (r_i^{-1})}{\partial x_j} = -r_i^{-1} \frac{\partial r_i}{\partial x_j}, \ \ \ \ \
\frac{\partial (wr_iw^{-1})}{\partial x_j}= w\frac{\partial r_i}{\partial x_j}
, \ \ \ \ \ \ \frac{\partial (r_i r_{\ell})}{\partial x_j} =
\frac{\partial r_i}{\partial x_j} +r_i \frac{\partial r_\ell}{\partial x_j}. \]
Therefore,
the changes of the $K_1$-classes with respect to Ia and Ib
are $A_{\rho,k} \mapsto - A_{\rho,k}$, and
$A_{\rho,k} \mapsto (-1)^k \rho (w) A_{\rho,k}$, respectively.
Furthermore, recalling from Whitehead Lemma \cite[Lemma III.1.3.3]{Wei} that
any elementary matrix is 1 in $K_1$,
we can easily verify that
the changes of the $K_1$-values with respect to Ic and II
are $A_{\rho,k} \mapsto A_{\rho,k}$, and
$A_{\rho,k} \mapsto \rho (w^{-1}) A_{\rho,k}$, respectively.
Since $\rho(x_k) = \tau \in K_1(\mathcal{A}_{\kappa}( \!( \tau) \!))$,
we have $\rho (w) = \tau ^{n_w} \in K_1(\mathcal{A}_{\kappa}( \!( \tau) \!))$ for some $n_w \in \mathbb{Z}$ by assumption.
Hence, the observation in the quotient $K_1(\mathcal{A}_{\kappa}( \!( \tau) \!) )/ \pm \tau $ proves (ii).
\end{proof}

\section{Relation to the $K_1$-class in \cite{Nos}}\label{invrelS2}
In what follows, we let $L$ be a knot embedded in the 3-sphere $S^3$, i.e., $\# L =1,$ and choose a knot diagram on $\R^2$.

In the paper \cite{Nos}, the author constructed
another $K_1$-class, where the construction is inspired by \cite{Lin}.
This section shows (Theorem \ref{cor1155}) that this class and another $[\Phi_{\rho,k}]$ in \S \ref{defS2} are almost equal.
In this section, we assume that $L$ is a knot, and choose a meridian $\mathfrak{m} \in \pi_1(S^3 \setminus L) $ such that $\rho( \mathfrak{m}) = \tau $.

We recall the presentation \eqref{oo456} and the definition of the $K_1$-class below.
We choose a Seifert surface $F$ of genus $g$ and
a bouquet of circles $W \subset F$ such that $W$ is a deformation retract of $F$
and the inclusion $F \subset S^3$ is isotopic to the standard embedding $W \subset F$.
Take a bicollar $F \times [-1,1]$ of $F$ such that $ F \times \{ 0\} =F$.
Let $\iota_{\pm}: F \ra S^3 \setminus F$ be the embeddings whose images are $ F \times \{ \pm 1\}$. Take generating sets $W:=\{ u_1, \dots, u_{2g} \}$ of $\pi_1 F $ and
$ X:=\{ x_1, \dots, x_{2g} \} $ of $ \pi_1(S^3 \setminus F)$, and set
$y_i^{\sharp}:=(\iota_+)_* (u_i)$ and $z_i=(\iota_-)_* (u_i)$;
a von Kampen argument yields a presentation
\begin{equation}\label{oo456} \langle \ x_1, \dots, x_{2g}, \mathfrak{m} \ | \ r_i^F := \mathfrak{m}^{-1} y_i \mathfrak{m}\ z_i^{-1} \ \ \ (1 \leq i \leq 2g) \ \ \rangle, \end{equation}
of $ \pi_1(S^3 \setminus L)$. 
Notice that
$$ \frac{\partial r_i^F}{\partial x_j}= \mathfrak{m} \frac{\partial y_i }{\partial x_j}- r_i^F \frac{\partial z_i }{\partial x_j} ,$$
and consider the square matrix of the form
\begin{equation}\label{kkkkk} A_{F,W}:= \Bigl\{ \rho (\frac{\partial r_i^F}{\partial x_j} ) \Bigr\}_{ 1 \leq i,j \leq 2g} = \Bigl\{ \tau \rho ( \frac{\partial y_j }{\partial x_i} )- \rho ( \frac{\partial z_j }{\partial x_i} ) \Bigr\}_{ 1 \leq i,j \leq 2g} \in \mathrm{Mat}(2g \times 2g, \mathcal{A}_{\kappa}( \!( \tau) \!)). \end{equation}

Then, the paper \cite[\S 3]{Nos} defined a quotient group
\begin{equation}\label{oo}\mathcal{Q}_{\mathcal{A}, \kappa}:= K_1( \mathcal{A}_{\kappa}( \!( \tau) \!))/ K_1(\mathcal{A}) , \end{equation}
and if $A_{F,W}$ is an invertible matrix, we define {\it the $K_1$-Alexander invariant (with respect to $\rho$)} to be the $K_1$-class
$$ \Delta_{\rho}^{K_1} := [ \tau^{-g} A_{F,W} ] \in \mathcal{Q}_{\mathcal{A}, \kappa} . $$
On the other hand, if $A_{F,W}$ is not invertible, we define $\Delta_{\rho}^{K_1}$ to be zero.
The main theorem in \cite{Nos} shows that this $ \Delta_{\rho}^{K_1}$ does not depend on the choice of $F,W,$ and $ X $.

Then, the theorem of this section is as follows:
\begin{thm}\label{cor1155}
Let $G$ be a knot group with a Wirtinger presentation, as in Example \ref{exppp0}, and
let $\rho: \Z[\pi_1(S^3 \setminus L)] \ra \mathcal{A}_{\kappa}( \!( \tau) \!) $ satisfy ($\dagger$).

Then, $\Phi_{\rho,k} $ is invertible if and only if so is $A_{F,W}$.
In addition, if so, the $K_1$-class $[ \tau^{-g} A_{F,W} ] $ is equal to $[\Phi_{\rho, k}] $
in the quotient group $ \mathcal{Q}_{\mathcal{A}, \kappa} / \{ \tau^{n} \}_{n \in \Z}$.
\end{thm}


\section{Review of the Reidemeister torsion in a $K_1$-group}\label{invrelS3}

For the proof of Theorem \ref{cor1155}, the Reidemeister torsion appearing in a $K_1$-group
plays a key rule. We will review the Reidemeister torsion. 
The definition and properties are based on \cite{Mil} or \cite[\S I.3]{Tur}.

Consider an exact sequence of length 3
$$ C_*: 0 \ra C_3 \stackrel{\partial_3}{\lra} C_2 \stackrel{\partial_2}{\lra}C_1 \stackrel{\partial_1}{\lra} C_0 \ra 0, \ \ \ $$
where $C_*$ is a finitely generated free $R$-module ($C_3$ may be zero).
Let us choose a basis $\mathcal{C}_i \subset C_i$ for all $i$ with $C_i \neq 0$.
Assume that $B_i =\mathrm{Im}(\partial_{i+1}) \subset C_i$ is free,
pick a basis $\mathcal{B}_i$ of $B_i$ and a lift $\tilde{\mathcal{B}}_i$ of $\mathcal{B}_i $ to $C_i$.
By $ \mathcal{B}_i \tilde{\mathcal{B}}_{i-1}$ we mean the collection of elements given by $\mathcal{B}_i$ and $\tilde{\mathcal{B}}_{i-1}$.
Since $C_*$ is exact, $ \mathcal{B}_i \tilde{\mathcal{B}}_{i-1}$ is indeed a basis for $C_i$.
For bases $d,e$ of a complex, denote by $[d/e]$ the invertible matrix of a basis change, i.e., $[d/e]= (a_{ij})$ where $d_i = \sum_j a_{ij} e_j$.
Then we define the Reidemeister torsion of the based acyclic complex $(C_* ,\mathcal{C}_i)$ to be
$$ \mathcal{T}( C_* , \mathcal{C}_i):= [ \tilde{\mathcal{B}}_{2}/ \mathcal{C}_3 ][ \mathcal{B}_2 \tilde{\mathcal{B}}_{1}/ \mathcal{C}_2 ]^{-1} [ \mathcal{B}_1 \tilde{\mathcal{B}}_0/ \mathcal{C}_1 ]
[ \mathcal{B}_0 /\mathcal{C}_0 ]^{-1} \in K_1 (R). $$
Meanwhile, in the case where the $R$-modules $B_i$ are not free, 
Section 3 in \cite{Tur} gives a definition of the torsion of $C_*$.
Since this paper does not consider such a case in details, we omit the details. 

We will study homology groups in local coefficients.
Let $X$ be a connected CW complex.
Denote the universal covering space of $X$ by $\tilde{X}$.
We regard the chain complex of space, $C_*(\tilde{X})$, as a chain complex of right $\Z[\pi_1(X)]$-modules,
where the $\Z[\pi_1(X)]$-module structure is defined via covering transformations.
Given a ring homomorphism $\rho : \Z[\pi_1(X)] \ra R$,
we can therefore consider the chain
complex $C_*(X;R)=C_*( \tilde{X}) \otimes_{\Z[\pi_1(X)]} R$.
We denote its homology by $H_* (X;R)$.

We now suppose that
$X$ is of finite type and dim$X \leq 3$.
Then, if the homology $ H_i (X;R)$ is not zero for some $i$, we write $\mathcal{T}(X, \rho) =0.$
Otherwise,
denote the $i$-cells of $X$ by $\sigma_i^1, \dots, \sigma^{r_i}_i$, and choose an orientation for each cell
$\sigma_i^j$, and also pick a lift $\tilde{\sigma}_i^j$ for each cell $\sigma_i^j$ to the universal cover $\widetilde{X}$.
Since the set $ \{\tilde{\sigma}_i^1 ,\dots, \tilde{\sigma}_i^{r_i} \} $ makes a basis $\mathcal{C}_i$ of $ C_i(X;R)$, we can define
\begin{equation}\label{jjj542} \mathcal{T} (C_i(X;R), \{ \mathcal{C}_i \}) \in K_1(R). \end{equation}
Moreover, consider the quotient of $K_1(R)$ passage by the image $\{ \pm \rho(g) \}_{g \in \pi_1(X)} $.
Then, as is known, the class
\begin{equation}\label{jjj5423} \mathcal{T}(X , \rho ):= [\mathcal{T} (C_i(X;R), \{ \mathcal{C}_i \})] \in K_1(R)/ \pm \rho(\pi_1(X) ) \end{equation}
depends only on the simple homotopy type of $X$ and the homomorphism $\rho : \pi_1(X) \ra R $. 

Next, we give a procedure for computing $\mathcal{T}(X , \rho ) $, following \cite[p. 8]{Tur}.
Take subsets, $\xi _i \subset \{ 1,2,\dots, \mathrm{rank}(C_i) \}$
so that $\xi_0 =\emptyset$, and
denote $(\xi_0,\xi_1,\dots, \xi_m)$ by $\xi $.
For $i,j,k$, define the matrix $\{ a^i_{jk}\}_{j,k}$ by considering only the $\xi_i$-columns of $A_i$
and with the $\xi_{i-1}$-rows removed.
Such a matrix chain $\xi$ is called {\it a $\mathcal{T}$-chain} if $A_1(\xi), \dots, A_m(\xi)$ are square matrices.
The following is the generalization of Turaev's Theorem 2.2 to the noncommutative setting, and is stated in \cite[Theorem 2.1]{Fri}.

\begin{prop}[{cf \cite[Theorem 2.2]{Tur}. See also \cite[Theorem 2.1]{Fri}}]\label{kkkk}
Let $\xi$ be a $\mathcal{T}$-chain such that $A_i(\xi)$ is invertible for all odd $i$.
Suppose that every $B_i$ is free.
Then, $A_i(\xi$) is invertible for all even i if and only if $H_*(C_*(X;R))=0$.
If $H_*(C_*(X;R))=0$, then
$$\mathcal{T}(C_i(X;R), \{ \mathcal{C}_i\})=\varepsilon \prod_{i=0}^m A_i(\xi)^{(-1)^i} \in K_1(R) \ \ \mathrm{ for \ some \ } \varepsilon \in \pm 1.$$
\end{prop}
\subsection{Proof of Theorem \ref{cor1155}}\label{231}
\begin{proof}[Proof of Theorem \ref{cor1155}]
We first recall the facts on invariance of Reidemeister torsion.
Let $Y$ be a finite connected finite CW complex.
According to \cite[Theorem 9.1]{Tur},
if there is a cellular map $f :Y \ra X$ which is homotopy equivalent, and $ H_*(X;R)=0$ and the Whitehead group of $\pi_1 (Y)$ is zero, then
$ H_*( Y;R)=0$ and the associated torsion $ \mathcal{T} (Y, \rho ')$ is equal to $ \mathcal{T} (X, \rho )$ in
$K_1( R)/ \pm \rho( \pi_1(X)).$ Here, $\rho ' = \rho \circ f_*$.
Furthermore, if $X$ is homotopic to the knot complement $S^3 \setminus L$,
then the Whitehead group of $\pi_1 (X)$ vanishes by Waldhausen \cite{Wal}.
Therefore, it is enough for the proof to find $X$ and $Y$ whose homotopy type are $S^3 \setminus L$
such that $\mathcal{T} (X, \rho) \cdot (1 - \tau)$ equals $ \Phi_{\rho,m } $ and
$\mathcal{T} (Y, \rho) (1- \tau)$ equals $ [A_{F,W}]$

First, we take the CW-complex $X$ corresponding with the Wirtinger presentation. Namely,
$X$ consists of a single vertex, $m$ edges labeled by the generators $x_1, \dots, x_m$ and $(m-1)$ 2-cells attached by the relations $r_1, \dots, r_{m-1}$.
As is known, we can easily verify that $W$ is homotopic to $S^3 \setminus L$.
We regards these $x_1,\dots, x_m$ and $r_{1} ,\dots, r_{m-1}$ as cells of $X$.
The cellular complex is written in
\begin{equation}\label{ooo24} 0 \ra \bigoplus_{j=1}^{m-1} \mathcal{A}_{\kappa}( \!( \tau) \!) r_j \stackrel{\partial_2}{\lra} \bigoplus_{i=1}^{m} \mathcal{A}_{\kappa}( \!( \tau) \!)x_i \stackrel{\partial_1}{\lra}
\mathcal{A}_{\kappa}( \!( \tau) \!) \ra 0,\end{equation}
as left $\mathcal{A}_{\kappa}( \!( \tau) \!) $-modules where the boundary maps are given by
$$ \partial_2 (r_j)= \sum_i \rho \Bigl( \frac{\partial r_j}{ \partial x_i}\Bigr) x_i , \ \ \ \mathrm{and} \ \ \ \partial_1 ( x_i )= 1- \rho (x_i) .$$
Notice that the restriction of $\partial_1 $ on $\mathcal{A}_{\kappa}( \!( \tau) \!) x_m $ is $1- \rho (x_m)= 1 - w_m^{-1} \tau w_m$; hence, it is invertible. Therefore, the acyclicity of \eqref{ooo24} is equivalent to the invertibility of $ \Phi_{\rho, k}$.
If we set $ \xi_1 := \{ 1\}$ and $\xi_2 := \{ 1,\dots, m\}$
as in Proposition \ref{kkkk}, the Reidemeister torsion $\mathcal{T} (X, \rho) $ is given by $ \Phi_{\rho,m } \cdot (1 - \tau)^{-1}$, as required.

Next, we let $Y$ be the CW complex corresponding with the presentation \eqref{oo456} in a similar way.
Then, it is shown \cite[Section 2.3]{Tro} that this $Y$ is known to be homeomorphic to $S^3 \setminus L$, and the
chain complex is given by
$$ 0 \ra \bigoplus_{j=1}^{2g} \mathcal{A}_{\kappa}( \!( \tau) \!) r_j^F \stackrel{\partial_2}{\lra} \bigoplus_{i=1}^{2g} \mathcal{A}_{\kappa}( \!( \tau) \!) x_i \oplus \mathcal{A}_{\kappa}( \!( \tau) \!) \mathfrak{m} \stackrel{\partial_1}{\lra}
\mathcal{A}_{\kappa}( \!( \tau) \!) \ra 0,$$
where the boundary maps are given by
$$ \partial_2 (r_j)= (1 -\rho(y_j) ) \mathfrak{m}+ \sum_i \rho \Bigl( \frac{\partial r_j^F}{ \partial x_i} \Bigr) x_i , \ \ \ \mathrm{and} \ \ \ \partial_1 ( x_i )= 1- \rho (x_i), \ \ \ \partial_1 ( \mathfrak{m} )= 1 -\tau .$$
Notice that the restriction of $\partial $ on $\mathcal{A}_{\kappa}( \!( \tau) \!) \mathfrak{m} $
is $1- \tau $ invertible. 
Therefore, similarly, Proposition \ref{kkkk} implies that the Reidemeister torsion $\mathcal{T} (Y, \rho) $ is $ A_{F,W} \cdot (1 - \tau)^{-1}$ by definition, as required.
This completes the proof.
\end{proof}

\begin{proof}[Proof of Proposition \ref{prop11}]
It is proven in the paper \cite{Nos} that, $K$ is fibered, if and only if the matrix $A_{F,W}$ is shown to be invertible for any $\rho$.
Thus, Proposition \ref{prop11} immediately deduces the proof.
\end{proof}

\section{Relation to the torsions of cyclic coverings}\label{def732}
Throughout this section, we assume the existence of $m \in \mathbb{N}$ such that $ \kappa^m= \mathrm{id}_{\mathcal{A}}$.
In this case, we suggest a relation to the torsions of the regular $m$-fold cyclic covering spaces $E_L^m$.
Here the $m$-fold cyclic covering $ p:E_L^m \ra S^3 \setminus L$ is associated with the surjection $ \pi_1(S^3 \setminus L) \ra \Z/m $.

For the purpose, we begin reviewing a ring homomorphism in \cite{Nos}.
For $\ell \leq m$ and $a \in \mathcal{A}$, we let $ \mathcal{D}_\ell (a)$ be a diagonal $(\ell \times \ell)$-matrix of the form
\[ \mathcal{D}_\ell (a):= \left(\begin{array}{rrrr} \kappa^\ell (a) & 0 & \cdots & 0\\
0 & \kappa^{\ell -1}(a) & \cdots & 0 \\
\vdots & \vdots & \ddots & \vdots \\
0 & 0 & \cdots & \kappa^{1}(a)
\end{array}\right) , \]
and let $\mathbb{O}_{s,t} $ be the zero $(s \times t )$-matrix. We define a square matrix
\begin{equation}\label{ooo222}M_{\ell}(a):= \left(\begin{array}{rr} \mathbb{O}_{m - \ell ,\ell} & \ \mathcal{D}_{m-\ell} ( \kappa^{\ell +1}(a)) \\
\mathcal{D}_\ell (a) & \!\!\!\!\!\mathbb{O}_{\ell, m - \ell}
\end{array}\right) \in \mathrm{Mat}(m \times m ; \mathcal{A}) . \end{equation}
Consider the Laurent polynomial ring $\mathcal{A}_{\rm id }( \!( t) \!) $, where $t$ is a commutative indeterminate.
Then, the paper \cite{Nos} introduced a ring homomorphism $\Upsilon$ defined by setting
\begin{equation}\label{ooo2}\Upsilon: \mathcal{A}_{\kappa}( \!( \tau) \!) \lra \mathrm{Mat}(m \times m ; \mathcal{A}_{\rm id }( \!( t) \!)) ; \ \ \ \ \ \sum_{j=k} a_j \tau^j \longmapsto \sum_{j=k} M_{j}(a_j) t^j . \end{equation}
Hence, we can consider the pushforwards of the $K_1$-classes: $ \Upsilon_*( \mathcal{T}(S^3 \setminus L, \rho )) $ and $\Upsilon_*( A_{F,W}) $. Here, by the Morita invariance on $K_1$
$$ \mathcal{I}: K_1( \mathrm{Mat}(m \times m ; \mathcal{A}_{\rm id }( \!( t) \!)) ) \cong K_1(\mathcal{A}_{\rm id }( \!( t) \!)) , $$
the $K_1$-classes are regarded as quotient elements of $K_1(\mathcal{A}_{\rm id }( \!( t) \!) ). $

Meanwhile, notice that the image of the composite $\rho \circ p_* : \Z[ \pi_1(E_L^m)] \ra \mathcal{A}_{\kappa}( \!( \tau) \!) $
is contained in $ \mathcal{A}_{\kappa}( \!( \tau^m ) \!) = \mathcal{A}_{\rm id }( \!( t ) \!)$. Thus, we can define
the torsion $ \mathcal{T}(E_L^m, \rho \circ p_* ) $ in $K_1( \mathcal{A}_{\rm id }( \!( t ) \!) ) $.
\begin{thm}\label{thm346} Let $L$ be a knot.
Under the above situation, the torsion of $E_L^m $ is equal to the pushforward of the torsion of $S^3 \setminus L$, that is,
$$\mathcal{I} \circ \Upsilon_*( \mathcal{T}(S^3 \setminus L, \rho )) = \mathcal{T}(E_L^m, \rho \circ p_* ) \in K_1( \mathcal{A}_{\rm id }( \!( t ) \!) ) / \pm \rho \circ p_* (\pi_1( E_L^m )). $$
\end{thm}

\begin{proof}We first analyze $ \pi_1(E_L^m)$.
Recall from \eqref{oo456} the presentation
$$ \pi_1(S^3 \setminus L) \cong \langle \ x_1, \dots, x_{2g}, \mathfrak{m} \ | \ r_i^F := \mathfrak{m}^{-1} y_i \mathfrak{m}\ z_i^{-1} \ \ \ (1 \leq i \leq 2g) \ \ \rangle. $$
For $ k \in \Z/m $, let $ x_i^{(k)}$ be a copy of $ x_i$,
and $y_i^{(k)}$ be the word obtained by replacing $x_i$ by $ x_i^{(k)}$ in the word $y_i$.
We similarly define the word $ z_i^{(k)}$.
Then, by a Reidemeister-Schreier method (see, e.g., \cite{LS,Kab}),
$ \pi_1(E_L^m)$ is presented by
$$ \langle \ x_1^{(k)}, \dots, x_{2g}^{(k)}, \overline{\mathfrak{m}} \ \ \ ( k \in \Z/m) \ | \ r_i^{(k )} := \overline{\mathfrak{m}}^{-1} y_i^{(k)} \overline{\mathfrak{m}}\ (z_i^{(k+1)})^{-1} \ \ \ (1 \leq i \leq 2g, \ k \in \Z/m) \ \ \rangle, $$
and the injection $ p_*: \pi_1(E_L^m) \ra \pi_1(S^3 \setminus L)$ is represented by the correspondence
$ x_i^{(k)} \mapsto x_i, \ \overline{\mathfrak{m}} \mapsto \mathfrak{m}^m . $

Then, similarly to \eqref{ooo24}, we have the cellular chain complex of $C_* (E_L^m ; \mathcal{A}_{\rm id }( \!( t ) \!) )$ as
$$ 0 \ra \bigoplus_{j=1}^{m} \bigoplus_{k=1}^{2g} \mathcal{A}_{\rm id }( \!( t ) \!) r_j^{(k)} \stackrel{\partial_2}{\lra} \mathcal{A}_{\rm id }( \!( t ) \!)\overline{\mathfrak{m} } \oplus\bigoplus_{i=1}^{m} \bigoplus_{k=1}^{2g} \mathcal{A}_{\rm id }( \!( t ) \!) x_i^{(k)} \stackrel{\partial_1}{\lra}
\mathcal{A}_{\rm id }( \!( t ) \!) \ra 0,$$
where the boundary maps are given by
$$ \partial_2 (r_j^{(k)})= (1 -\rho(y_j^{(k)}) ) \overline{\mathfrak{m} } + \sum_{i,k'} \rho \Bigl( \frac{\partial r_j^{(k)}}{ \partial x_i^{(k')}} \Bigr) x_i^{(k')} , \ \ \ \mathrm{and} \ \ \ \partial_1 ( x_i^{(k)} )= 1- \rho (x_i^{(k)}), \ \ \ \partial_1 ( \overline{\mathfrak{m} } )= 1 -t .$$
Let $\mathcal{J}$ be the $(2gm \times 2gm)$-matrix $\{ \rho \bigl( \partial r_i^{(k )}/ \partial x_j^{(k' )} \bigr) \}_{i,j \leq g, \ \ \ k,k' \leq m}$. Then, similarly to the proof of Theorem \ref{cor1155},
Proposition \ref{kkkk} implies
\begin{equation}\label{ooo92} \mathcal{T}(E_L^m, \rho \circ p_* ) = (1-t^m)^{-1} \cdot \mathcal{J} \in K_1( \mathcal{A}_{\rm id }( \!( t ) \!) ) / \pm \rho \circ p (\pi_1( E_L^m )). \end{equation}

On the other hand, by the definition of $\Upsilon$, we can easily check
$$ \Upsilon ( \rho (\frac{\partial r_i^F}{ \partial x_j})) = \Bigl\{\rho ( \frac{\partial r_i^{(s )}}{\partial x_j^{(t )}}) \Bigr\}_{1 \leq s,t \leq m} \in \mathrm{Mat}(m \times m ; \mathcal{A}_{\rm id }( \!( t) \!)) , $$
for any $i,j \leq m. $
Therefore, as $(2gm \times 2gm)$-matrices, $ \Upsilon ( A_{F,W}) = \mathcal{J}$, where $A_{F,W}$ is the matrix in \eqref{kkkkk}.
Notice that $ \Upsilon ( 1-\tau) = 1-t^m$ in $ K_1( \mathcal{A}_{\rm id }( \!( t ) \!) ) $,
and recall $ \mathcal{T}(S^3 \setminus L, \rho ) = (1-\tau)^{-1} A_{F,W} $ by the proof of Theorem \ref{cor1155}. Hence, combing those with \eqref{ooo92} deduces the required equality.
\end{proof}

\section{Reciprocity of some $K_1$-Alexander invariants}\label{invdefS342}
As is well-known, the (classical) Alexander polynomial of a knot has symmetry, i.e., it can be expanded as $\sum_{ i= -m}^{m} a_i t^{i}$ such that $ a_i \in \Z$ and $ a_{-i} = a_i$.
Such a symmetry is called {\it Reciprocity}.
As a generalization, if we choose an appropriate representation,
the twisted Alexander polynomial also has reciprocity in some sense; see, e.g., \cite{FKK,FV,Kit,Kitano}.
Moreover, reciprocity of some Reidemeister torsions is generalized for $K_1$-groups; see
\cite[\S 10]{Mil} or \cite[Theorem 14.1]{Tur}.
Thus, it is reasonable to ask reciprocity on the $K_1$-classes $\Delta_{\rho}^{K_1}$.
Unfortunately, some papers on reciprocity require some conditions to show duality theorems.
Similarly, the author could not show directly reciprocity for every $\Delta_{\rho}^{K_1}$; however, this paper will observe (Theorem \ref{pro366}) a reciprocity of the pushforward by the ring homomorphism $\Upsilon_*$, under a certain situation.
Here the point is to find a situation applicable to the theorem of Turaev.

The situation in this section is described as follows:

\

($\star$) As in Example \ref{exppp0},
we consider a group homomorphism $h: \pi_1(S^3 \setminus L) \ra H \rtimes \Z$ such that $H$ is a group of finite order, and $h (x_i) =(g_i, 1) $ for some $g_i \in H$.
Let $\mathcal{A} $ be the group ring $\Q [H]$ over $\Q$.
Let $m \in \mathbb{Z}_{\geq 0}$ be the minimal such that $ \kappa^m= \mathrm{id}_{\mathcal{A}}$.

\

Starting from the situation $(\star)$, we discuss some involutions and observe some $K_1$.
In this section, let $ \mathcal{A}$ be $\Q[H] $, and consider the (skew) Laurent polynomial ring $ \mathcal{A}_{\kappa}[ \tau^{\pm 1} ] $, instead of Novikov rings.
Since $ \mathcal{A}= \Q[H] $ is semi-simple,
the Wedderburn theorem immediately implies that there are division rings $D_1, \dots, D_m$ over $\Q$ and integers $ n_1, \dots, n_k$ which ensure the ring isomorphism
$$ \Q[H] \cong \bigoplus_{i : 1 \leq i \leq k} \mathrm{Mat}(n_i \times n_i ; D_i). $$
For a division ring $D$, let $ D(t)$ be the fractional field of $D[t^{\pm 1} ]$ with $\bar{t}=t^{-1}$
and $\iota_D: D[t^{\pm 1} ] \ra D(t)$ be the inclusion.
Then, the direct sum $ \oplus \iota_{D_i}$ gives rise to
$$ (\oplus \iota_{D_i})_*: \mathrm{Mat}(m \times m ; \mathcal{A}_{\rm id }[ t^{\pm 1} ]) \lra \bigoplus_{i : 1 \leq i \leq k} \mathrm{Mat}(n_i m \times n_im ; D_i (t)). $$
Formally, we will denote by $\Q[H](t)$ by the direct sum $ \oplus_{i=1}^k D_i(t)$. Then, Morita invariance on $K_1$-groups again implies the isomorphisms
$$ K_1( \bigoplus_{i : 1 \leq i \leq k} \mathrm{Mat}(n_i m \times n_im ; D_i (t))) \cong
\bigoplus_{i : 1 \leq i \leq k} K_1 (D_i (t) ) \cong K_1( \Q[H](t)) . $$
By the assumption ($\star$), $A_{F,W}$ in \eqref{kkkkk} is regarded as
a matrix over the (uncompleted) Laurent polynomial ring $ \Q[H]_{ \kappa} [\tau^{\pm 1}]$.
In the same way as \eqref{ooo2}, consider the ring homomorphism $\Upsilon$ defined by setting
\begin{equation}\label{ooo299}\Upsilon: \mathcal{A}_{\kappa}[ \tau^{\pm 1} ] \lra \mathrm{Mat}(m \times m ; \mathcal{A}_{\rm id }(t^{\pm 1} )) ; \ \ \ \ \ \sum_{i=k} a_i \tau^i \longmapsto \sum_{i=k} M_{i}(a) t^i , \notag \end{equation}
where $M_j(a_j) $ is defined in \eqref{ooo222}.
Then, the pushforward $(\oplus \iota_{D_i})_*\circ \Upsilon_*(A_{F,W}) $ can be regarded over $\Q[H](t)$.
Thus, $\Upsilon_*(A_{F,W})$ can be considere to be an element of the $K_1(\Q[H](t))$, where we omit writing $  (\oplus \iota_{D_i})_* $.
\begin{thm}\label{pro366} 
Let $\mathfrak{l} \in \pi_1(S^3 \setminus L)$ be the preferred longitude of the knot $L$.
Suppose the above situation ($\star$), and invertibility of the pushforward $\Upsilon_*(A_{F,W}) $ over $\Q[H](t) $. In addition, we assume that either (a) $ \Upsilon ( 1- \rho ( \mathfrak{l} )) $ is an invertible matrix or (b) $ \rho (\mathfrak{l} )= 1 \in \Q[H] $.

Then, under the isomorphism $  (\oplus \iota_{D_i})_* $, the following equality holds:
$$ \overline{\Upsilon_*( A_{F,W}) }= \Upsilon_*(A_{F,W}) \in \frac{K_1(\Q[H](t) )}{\{ K_1(\Q[H]) , t^{\ell} \}_{\ell \in \Z}}.
$$
\end{thm}
\begin{rem} As a result, 
we can observe the reciprocity on the $K_1$-class $\Upsilon_*(\Delta^{K_1}_{\rho}) $ in the $K$-group of a Novikov ring $ \mathcal{Q} _{\Q[H]( \!( t) \!) , \mathrm{id}}$ as well.
Indeed, since $D_i(t)$ and $D_i ( \!( t) \!)$ are division rings, the associated inclusions $ j_{D_i}: D_i(t)\ra D_i( \!( t) \!) $ induce
$\iota: \Q[H](t) \ra \Q[H] ( \!( t) \!)$; thus, the reciprocity inherits on $ \iota_* (\Upsilon_*( A_{F,W}) )$ in $ \mathcal{Q} _{\Q[H]( \!( t) \!) , \mathrm{id}}$.

Furthermore, since the Dieudonn\'{e} determinant over
any division ring $D$ induces an isomorphism $K_1(D) \cong (D^{\times})_{\rm ab}$, the $K_1$-group $ K_1(\Q[H](t))$ is identified with $ (\Q[H](t)^{\times})_{\rm ab}$.
In particular, if $H$ is abelian, we can interpret Theorem \ref{pro366} as the reciprocity of a polynomial.
\end{rem}

\begin{proof}[Proof of Theorem \ref{pro366}]
Let $E_L^m$ be the cyclic covering space as above. For simplicity, let $\Upsilon_{\rho} $ denote the composite $ \Upsilon \circ \rho$.
By virtue of Theorem \ref{thm346}, we may show only the reciprocity in Reidemeister torsions of $E_L^m$.

We first give the proof in the case (a).
Then,
\cite[Corollary 14.2]{Tur} immediately says
\begin{equation}\label{ooo111111}\mathcal{ T} ( E_L^m, \partial E_L^m; \Upsilon_{\rho} ) = \overline{ \mathcal{ T} ( E_L^m ; \Upsilon_{\rho} )} \in K_1(\Q[H](t) )/ \{ K_1(\Q[H]) , \tau^{\ell } \}_{\ell \in \Z}.
\end{equation}
Moreover, thanks to the multiplicity of torsions with respect to short exact sequences (see \cite[Theorem 3.4]{Tur}), we immediately have
\begin{equation}\label{ooo111}\mathcal{ T} ( E_L^m, \partial E_L^m; \Upsilon_{\rho} ) \mathcal{ T} ( \partial E_L^m; \Upsilon_{\rho} ) =
\mathcal{ T} ( E_L^m; \Upsilon_{\rho} ) \in K_1(\Q[H](t) )/ \{ K_1(\Q[H]) , t^{\ell} \}_{\ell \in \Z}.
\end{equation}
By Lemma \ref{lem99} below, the second term $ \mathcal{ T} ( \partial E_L^m; \Upsilon_{\rho} ) $ is equal to $1$. Hence, the equalities \eqref{ooo111111} and \eqref{ooo111} readily mean the required equality.

Next, we discuss the case (b).
Let $M_0$ be the closed 3-manifold obtained by $0$-surgery of $E_L^m $ along $L$.
Notice $\partial M_0 = \emptyset$. Let $p: E_L^m \ra S^3 \setminus L $ be the covering. 
Since $ \rho (\mathfrak{l} )= 1$, the composite $ \rho \circ p_*$ induces
$\Z[\pi_1(M_0 )] \ra \mathcal{A}_{\kappa }( \!( \tau ) \!) . $
Therefore, by \cite[Corollary 14.2]{Tur} again, we readily have the reciprocity
$ \mathcal{ T} ( M_0 ; \Upsilon_{\rho \circ p_*} ) = \overline{ \mathcal{ T} ( M_0 ; \Upsilon_{\rho \circ p_*} )}. $
Hence, it is enough for the proof to show $ (1-t) \mathcal{ T} ( M_0 ; \Upsilon_{\rho \circ p_*} ) = \mathcal{ T} ( E_L^m; \Upsilon_{\rho} ). $

To show this, consider the exact sequence appearing in a Mayer-Vietoris argument
$$ 0 \lra C_*(S^1 \times S^1 ) \stackrel{\iota_1 \oplus \iota_2 }{\lra} C_* (D^2 \times S^1 ) \oplus C_*(E_L^m ) \lra C_*( M_0) \lra 0 \ \ \ \ \ \ (\mathrm{exact}). $$
Since $ 1- \Upsilon_{\rho} ( \mathfrak{l} )=0$, the boundary maps in the second $C_* (D^2 \times S^1 )$ are zero, and $\iota_1$ has a splitting as a chain map. Hence, the sequence is reduced to
$$ 0 \lra C_*(S^1 ) \stackrel{ \iota_2 }{\lra} C_*(E_L^m ) \lra C_*( M_0) \lra 0 \ \ \ \ \ \ (\mathrm{exact}). $$
Since the first and second complexes are acyclic, so is the third.
Notice $\mathcal{T}(C_*(S^1 ) )= (1-t)$ by definition. Hence, by \cite[Theorem 3.4]{Tur} again,
we obtain $(1-t) \mathcal{ T} ( M_0 ; \Upsilon_{\rho \circ p_*} ) = \mathcal{ T} ( E_L^m; \Upsilon_{\rho} ) $ as required.
\end{proof}

\begin{lem}\label{lem99} If $ 1 -\Upsilon_{\rho}(\mathfrak{l}) \in \Q[H]$ is invertible, then the torsion $ \mathcal{ T} ( \partial E_L^m; \Upsilon_{\rho} ) $ is 1.
\end{lem}
\begin{proof}Since $\partial E_L^m$ is torus, we can describe the cellular complex as
$$0 \lra \Q[H](t)c_2 \xrightarrow{ (1- \Upsilon_{\rho} (\overline{\mathfrak{m}} ) , 1- \Upsilon_{\rho} (\mathfrak{l}) ) }\Q[H](t)c_1 \oplus \Q[H](t)c_1'
\xrightarrow{\small \begin{pmatrix}
\Upsilon_{\rho} (\mathfrak{l})-1 \\
1- \Upsilon_{\rho} (\overline{\mathfrak{m}} ) \\
\end{pmatrix} }
\Q[H](t) c_0 \lra 0 ,$$
which is acyclic by assumption. If we choose
$$ b_2 = \{ c_2\}, \ \ \ b_1 = \{ c_1\}, \ \ \ b_0 = \{ c_0\},$$
then
$$ [b_2/c_2]=1, \ \ [b_1 \partial_2 b_2 /c_1] = \begin{pmatrix}
1- \Upsilon_{\rho} (\mathfrak{l}) & 1\\
1- \Upsilon_{\rho} (\overline{\mathfrak{m}} ) & 0 \\
\end{pmatrix}, \ \ \ \ \ \ [b_0 \partial_1 b_1 /c_0] = 1- \Upsilon_{\rho} (\overline{\mathfrak{m}} ) .$$
Hence, by the definition of $\mathcal{T}$, we get the conclusion.
\end{proof}

\section{Applications to $m$-fold metabelian Alexander polynomials.}\label{ExaSS77}
The paper \cite{Nos} introduces an ($m$-fold) metabelian Alexander polynomial.
As applications of the previous sections, we will see properties of
the polynomial.

Let us recall the polynomial.
Let $L$ be a knot, 
and $H$ be the torsion subgroup $\mathrm{Tor}(H_1(\widetilde{E}_L^m ; \Z)) $.
Suppose that $ m H =H$, e.g., the case $m$ is a prime power.
Then, the commutator subgroup of $\pi_1(S^3 \setminus L ) $ surjects on $H$ via an abelianization.
Thus, we canonically have a group epimorphism $\rho_{m}^{\rm meta} : \pi_1(S^3 \setminus L ) \ra H \rtimes \Z $ satisfying the assumption ($\dagger$).
Notice $ \mathcal{A}= \Q[H] $ is 
isomorphic to a direct sum of cyclotomic fields over $\Q$; see, e.g., \cite[Corollary 12.10]{Tur}.
Then, {\it the metabelian Alexander polynomial} is defined to be the determinant of the pushforward $\Upsilon_*(A_{F,W})$:
$$ \Delta_{\rho_{m}^{\rm meta} }:= \mathrm{det} (\Upsilon_*(A_{F,W})) \in \Q[H] (t)^{\times }/ (\Q[H]^{\times}). $$
As seen in Appendix \ref{paji}, this polynomial can be also interpreted as a torsion obtained from a $S^1$-valued Morse theory.

As a slight generalization of Lemma 4 of \cite{CG}, we will show the invertibility of $\Upsilon_*(A_{F,W}) $.
\begin{prop}[{cf. \cite{CG}}]\label{iiii}
The matrix $\Upsilon_*(A_{F,W})$ in $ \mathrm{Mat}( 2gm \times 2gm; \Q[H] (t) )$ is invertible.
\end{prop}
\begin{proof} Let $ S \subset H= \mathrm{Tor} H_1( E_L^m ;\Z)$ be
any direct summand isomorphic to $\Z/p^n$ for some $n \in \mathbb{N}$ and prime $p$.
For the proof, we may replace $H$ by $S$, since 
$GL_n(\Q[H_1 \oplus H_2] )= GL_n( \Q[H_1])\times GL_n (\Q[H_2])$ for any finite abelian groups $H_i.$

Let $\widetilde{E}_L$ be the infinite cyclic covering space of $S^3 \setminus L$. Let $\widetilde{M} \ra \widetilde{E}_L $ be the abelian finite covering associated
with the projection $\pi_1( \widetilde{E}_L) \stackrel{\rm proj.}{\lra} \mathrm{Tor} H_1( E_L^m ;\Z) \stackrel{\rm proj.}{\lra} S$. According to Lemma 4 and its corollary of \cite{CG}, the rational homology $H_*( \widetilde{M}; \Q )$ is shown to be zero.
Moreover, by Shapiro lemma (see, e.g., \cite{KL}), the cellular complex $C_*( \widetilde{M}; \Q )$ is isomorphic to $C_*(E_L^m ; \Q[S][t^{\pm 1}] ) $ of finite dimension. 
Thus, the tensored complex $ C_*(E_L^m ; \Q[S](t) ) = C_*(E_L^m ; \Q[S][t^{\pm 1}] ) \otimes \Q[S](t) $ over $\Q[S](t) $ is acyclic.
In usual, the first boundary $\partial_1$ is a splittable surjection; hence,
the acyclicity of $ C_*(E_L^m ; \Q[S](t) ) $ implies the invertibility of $\partial_2 $. As seen in the proof of Theorem \ref{cor1155}, $\partial_2$ is represented by the matrix $\Upsilon_*(A_{F,W}) $. Hence, we complete the proof.
\end{proof}
As a corollary, we will see the reciprocity of $ \Delta_{\rho_{m}^{\rm meta} }$.
Here, notice that $ \rho(\mathfrak{l})=0_H$ since the preferred longitude $\mathfrak{l}$ is bounded by a Seifert surface in $\widetilde{E}_L$.
Therefore, our situation fulfills the conditions in Theorem \ref{pro366}; hence, we immediately have
\begin{cor} The metabelian Alexander polynomial has reciprocity in the sense of
$$\overline{ \Delta_{\rho_{m}^{\rm meta} }} = \Delta_{\rho_{m}^{\rm meta} } \in \Q[H] (t)^{\times }/\{ \Q[H]^{\times}, \ t^{\pm 1} \}. $$
\end{cor}
This situation is similar to the Casson-Gordon invariant \cite{CG}, which discusses obstructions of sliceness.
Thus it is reasonable to consider applications for sliceness from our results.
\begin{thm}[{cf. \cite[Theorem 6.2]{KL}}]\label{lll498}Suppose that the knot $L$ is topological slice, and $H_1(E_L^m ;\Z ) $ is isomorphic to $\Z \oplus T $ for some torsion module $T$.

Then, there is a subgroup $ B \subset H =\mathrm{Tor}H_1(E_L^m ;\Z )$ such that $
| B|^2= | H |$ 
and that there exists a polynomial $f(t) \in \Q[H/B](t) $ satisfying
\begin{equation}\label{ooo001}\mathcal{P}_*( \Delta_{\rho_{m}^{\rm meta} }) =a t^n (1-t)f(t)\overline{f(t)} \in \Q[H/B](t),
\end{equation}
for some $ a \in \Q[H/B]^{\times} , n \in \mathbb{Z}$,
where $\mathcal{P}: H \ra H/B $ is the projection.
\end{thm}

\begin{rem}\label{l98}
We give a comparison with the works \cite{KL,HKL}.
The papers suppose $m$ to be a prime power, and consider a homomorphism $\chi: H_1(B_L^m;\Z) \ra \Z/p^d \Z$ for some prime $p$. Choose a homomorphism $\lambda: \Z/p^d \Z \ra GL_1(\mathbb{C})$.
Then, \cite[Theorem 6.2]{KL} claims that
the further pushforward $ (\lambda \circ \chi)_*\circ \mathcal{P}_*( \Delta_{\rho_{m}^{\rm meta} }) $ is decomposed as \eqref{ooo001}. Thus, Theorem \ref{lll498} is a slight generalization of \cite[Theorem 6.2]{KL}, although the proof below is outlined on discussions in \cite{CG,HKL,KL}.
\end{rem}

\subsection{Proof of Theorem \ref{lll498}}
For the proof, we review Reidemeister torsions suitable to non-acyclic cases, in terms of determinants.
Let $\mathbb{F}$ be a commutative field of characteristic zero.
Let $X$ be a finite connected CW-complex $X$. 
Choose homomorphisms $ \rho: \pi_1(X) \ra \mathbb{F}^{\times}$ and
$\alpha: \pi_1(X) \ra \mathbb{Z}= \{ t^n \}_{n \in \Z}$.
We have the tensor representation $ \rho\otimes \alpha : \pi_1(X) \ra \mathbb{F}[t^{\pm 1}]^{\times}$.
Then, the cellular complex with local coefficients is defined to be
$$C_*(X; \mathbb{F}(t)) = \mathbb{F}(t) \otimes_{ \mathbb{F}[\Z]} C_*(X; \mathbb{F}[\Z] ) =\mathbb{F}(t) \otimes_{ \rho \otimes \alpha} C_*(\widetilde{X}).$$
If $C_*$ is a based chain complex, $c_i$ is a basis for $C_i$, $b_i$ a basis for the boundaries $B_i$, $h_i$ a basis for the homology $H_i$,
then {\it the Reidemeister torsion $\mathcal{T}'$ of the based chain complex} is defined by
$$ \mathcal{T}(X,\rho, h_*)'= \frac{\prod_i \mathrm{det}[ b_{2i} \widetilde{h}_{2i}\widetilde{b}_{2i} / c_{2i}]}{ \prod_i \mathrm{det}[
b_{2i-1} \widetilde{h}_{2i-1}\widetilde{b}_{2i-1} / c_{2i-1}] } \in \mathbb{F}(t)^{\times} /\{ t^{n}, a \in \mathbb{F}^{\times} \}. $$
In this expression, $ \widetilde{h}_{i}$ is a choice of lift of $h_i$ to $C_i$, and $\widetilde{b}_{i} $ is a choice of $b_i$ to $C_{i+1}$ using the differential $\partial_{i+1} : C_{i+1} \ra C_i. $
It is known this $\mathcal{T}(X,\rho,h_*)'$ is independent of the choices of $b_i$ and of the lifts.
Notice that, if $C_*$ is acyclic and $\mathrm{dim}X \leq 3$, then
this $ \mathcal{T}(X,\rho, \emptyset )'$ is recovered by
the torsion in \S \ref{invrelS3}: precisely,
$\mathrm{det} \mathcal{T}(X,\rho \otimes \alpha)= \mathcal{T}(X,\rho, \emptyset )'$ by definitions. 

In order to start the proof of Theorem \ref{lll498},
following \cite{Mil,Mil2}, let us notice the following theorem (which is also stated in \cite[Theorem 5.1]{KL}).
\begin{thm}[{see \cite[Theorem 5.1]{KL}}]\label{lll354}
Let $X$ be an oriented connected compact $C^{\infty}$-manifold with boundary $\partial(X)$. Choose a triangulation of $X$ as a CW-complex.
Suppose that $C_*(X;\F (t))$ is not acyclic, but $C_*(\partial X; \F (t))$ is acyclic, and choose an appropriate basis $h_q$ for $H_q(X;\F (t))$ for each $q$. 
Then, with respect to these bases,
$$ \mathcal{T}(\partial X, \rho)'=\mathcal{T}( X, \rho, h_*)' \overline{\mathcal{T}( X, \rho,h_*)' }^{ (-1)^{\mathrm{dim} (X)}}. $$
\end{thm}

\begin{proof}[Proof of Theorem \ref{lll498}]
We give some preparations.
Let $D \subset B^4$ be the slice disk of $L$ in the 4-ball,
and $\widetilde{B} \ra B^4 \setminus D$ be the $ m$-fold covering.
Let $M_0$ (resp. $X$) be the oriented closed manifold obtained by $0$-surgery of $E_L^m $ along $L$ (resp. of $\widetilde{B} $ along $D$). Then, we notice $ \partial X =M_0$, and
$\mathrm{Tor}H_1(M_0;\Z) \cong H= \mathrm{Tor} H_1( \widetilde{E}_L^m ;\Z) $.
Let $B \subset H$ be the subgroup consisting of metabolizers of the linking form on $M_0$.
Then, it is known (see \cite{CG} or \cite[\S 6]{KL}) that $|B|^2=|H|$ and the projection $\rho_{m}^{\rm meta} : \pi_1(M_0) \ra H$
extends to $ \pi_1(X) \ra H/B$.
Recall that the group ring $\Q[H/B]$ is ring isomorphic to $\oplus_i \mathbb{F}_i $, where $\mathbb{F}_i $ is a cyclotomic field over $\Q$.
Let $q_i : \Q[H/B] \ra \mathbb{F}_i$ be the projection.

Notice that we can find $f_i \in \mathbb{F}_i(t) $
satisfying 
\begin{equation}\label{ooo236}q_i \circ \mathcal{T}(M_0, \rho)' =a_i t^{n_i} f_i(t)\overline{f_i(t)} \in \mathbb{F}_i(t),
\end{equation}
for some $ a_i \in \mathbb{F}_i^{\times} , n_i \in \mathbb{Z}$.
Indeed, we may only let $ \rho_i$ be $q_i \circ \rho_{m}^{\rm meta} $, and
$ f_i (t)$ be $ \mathcal{T}( X, \rho_i, h_*) $ in Theorem \ref{lll354}. Moreover, as in Theorem \ref{pro366}, each $ \mathcal{T}(M_0, \rho_i)' $ has reciprocity, we may suppose $n_1=n_2= \cdots $.
Furthermore, by the end of the proof of Theorem \ref{pro366}, we
notice $ q_i \circ \mathcal{P}_* (\Delta_{\rho_{m}^{\rm meta} })=\mathcal{T}(E_L^m, \rho_i)' =(1-t) \mathcal{T}(M_0, \rho_i)' $.
To conclude, since $\Q[H/B] (t ) \cong \oplus_i \mathbb{F}_i (t) $,
the multiplications of \eqref{ooo236} running over $i $ implies the required equality \eqref{ooo001}.
\end{proof}


\appendix
\section{Appendix; relation to $S^1$-valued Morse theoretic torsions}\label{paji}

Fix a preferred longitude $\mathfrak{l} \in \pi_1(S^3 \setminus L)$.
Under an assumption $\rho (\mathfrak{l} )=1$,
the $K_1$-class $\Phi_{\rho, k}$ can be described from
$S^1$-valued Morse theory; however,
this section essentially contains nothing new. In fact, this sections are analogous to \cite{Paji,Kit2}.
Here, we suppose notation in \S\S \ref{defS2} and \ref{invrelS3}.

Let us roughly review the (twisted) $S^1$-valued Morse theory, and the main result of \cite{Paji}; see \cite{Paji} for the details.
Let $M$ be a connected closed $C^{\infty}$-manifold, and
$f : M \ra S^1 $ be a Morse map. Suppose that $f_* : H_1(M) \ra H_1(S^1)= \Z$ is surjective.
As in Example \ref{exppp0}, let $G$ be a semidirect product $ H \rtimes \Z$ with projection $\chi : G \ra \Z$, and
$ \mathcal{A}$ be $ \Q[H]$. Then, we have the Novikov ring $\mathcal{A}_{\kappa}( \!( \tau) \!) $.

We study a logarithm from $K_1( \mathcal{A}_{\kappa}( \!( \tau) \!) )$.
For $n \in \mathbb{Z}$, let $\Gamma_n$ be the set of conjugacy classes contained in $ \chi^{-1}(n)$.
Take the $\Q$-vectors space $\Q \Gamma_n$ spanned by $\Gamma_n$, and
define $ \mathfrak{G}$ by $\prod_{n \geq 0} \Q \Gamma_n$.
Furthermore, consider the multiplicative subgroup, $W$, of $ \mathcal{A}_{\kappa}( \!( \tau) \!)^{\times } $ consisting of elements of
the form $1+ a_1 \tau + a_2 \tau ^2 + \cdots$. The image of $W$ in $K_1( \mathcal{A}_{\kappa}( \!( \tau) \!) )$ will be denoted by $\widehat{W}$.
The logarithm $\mathrm{log}: W \ra \mathfrak{G} $ is defined by
$$ \mathrm{log}(1 + \mu \tau ) = \mu \tau - \frac{(\mu \tau)^2}{2} + \cdots + (-1)^{n-1} \frac{(\mu \tau)^n}{n} + \cdots \ \ \ \ \mathrm{with} \ \ \mu \in \mathcal{A}_{\kappa}[ \! [ \tau] \!] .$$
Then, it is shown \cite[Lemma 1.1]{Paji} that this map induces a homomorphism $ \mathcal{L}: \widehat{W} \ra \mathfrak{G}$.
The subgroup $\widehat{W} $ is known to be a direct summand of $K_1( \mathcal{A}_{\kappa}( \!( \tau) \!) )$ \cite{PR}; hence, $ \mathcal{L}$ is regarded as a homomorphism from $ K_1( \mathcal{A}_{\kappa}( \!( \tau) \!) )$.

Next, we will review non-ablelian eta functions.
For a vector field $v$, which satisfies a ``Kupka-Smale condition",
let $\mathrm{Cl}(v)$ be the closed orbits of $v$.
For a closed orbit $\gamma \in \mathrm{Cl}(v)$, let $\epsilon(\gamma) \in \{\pm 1 \}$ denote the
index of the corresponding Poincar\'{e} map; let $m(\gamma)$ denote the multiplicity of $\gamma$.
Here is the definition of {\it the non-ablelian eta function} of $-v$:
$$ \eta_L(-v) = \sum_{\gamma \in \mathrm{Cl}(-v) } \frac{\epsilon (\gamma)}{ m (\gamma)} \{ \gamma\} \in \mathfrak{G} .$$
Let denote $ \mathcal{G}(f)$ be the set of vector fields satisfying the Kupka-Smale condition.

The main theorem of \cite{Paji} clarified the difference between torsions from the usual complex and from the Novikov complex.
To be precise, 
there is an open dense subset $ \mathcal{G}_0(f) \subset \mathcal{G}(f)$ with respect to a certain $C^0$-topology such that, for every $v \in \mathcal{G}_0(f)$, there is a chain homotopy equivalence
$$ \phi: C_*(v) \lra C_*( \widetilde{M};\Z)\otimes_{\Z[H] } \mathcal{A}_{\kappa}( \!( \tau) \!) $$
such that 
$$ \mathcal{L} (\tau (\phi)) = \eta_L( -v). $$
Here $\tau(\phi)$ is the torsion of the acyclic complex of the cokernel $\mathrm{Coker}(\phi)$.

We will give a conclusion from the assumption ($*$) and
suppose $\rho (\mathfrak{l})=1$ as above (For instance, the setting in \S \ref{ExaSS77} is satisfied).
Let $M_{L,0}$ be the closed 3-manifold obtained by $0$-surgery of $S^3 \setminus L $ along $L$.
Then, $\rho$ induces a homomorphism $ \pi_1( M_{L,0}) \ra \mathcal{A}_{\kappa}( \!( \tau) \!) $.
As in the proof of Theorem \ref{pro366} implies the torsion $\mathcal{T} (M_{L,0}, \rho ) $ is equal to $(1-t)^{-1} \mathcal{T} (S^3 \setminus L, \rho )$. Then, by the proof Theorem \ref{cor1155}, we have
$ [\Phi_{\rho,k}]= (1-t)\mathcal{T} (M_{L,0}, \rho ) . $
To conclude, the torsions $\mathcal{L}([\Phi_{\rho,k}] ) $} can be described by the torsion defined from $S^1$-valued Morse complex and the eta function.

\section{Relation to the higher order Alexander polynomial}\label{invrelS4}

The previous papers \cite{C,Har} suggested a generalization of the classical Alexander polynomial; see also \cite{GS,Fri2} for the studies.
The point is that the generalized polynomial is defined from
a group $G$, which is locally indicable and amenable, not any group. 
The purpose of this section is to show Proposition \ref{prop188} below, which suggests a relation between the generalized polynomial and the $K_1$-class $\Delta_{\rho}^{K_1}$. The facts and explanations in this section are essentially based on \cite{Fri2}; there is almost nothing new in this appendix. This appendix supposes terminology in \S\S \ref{defS2}-- \ref{invrelS3}.

For this, we start by reviewing the situation in \cite{C,Har}.
Let $\mathrm{Ab}: \pi_1(S^3 \setminus L) \ra \Z $ be an abelianization.
Let $G$ be a locally indicable and amenable group. It is known that the group ring $\Z[G]$ embeds in a (skew) fractional field. 
Fix an epimorphism $\phi: \pi_1(S^3 \setminus L) \ra G$ 
such that there exists a group homomorphism $\varphi_G: G \ra \Z$ satisfying $\mathrm{Ab}= \varphi_G \circ \phi $.
Such a pair $(\varphi_G, \phi)$ is called {\it an admissible pair} for $\pi_1(S^3 \setminus L) $, following \cite[Definition 1.4]{Har}.

Let $G'$ be the kernel $\Ker (\varphi_G : G \ra \Z)$. Since $G' $ is also locally indicable and amenable, $\Z[G']$ embeds in a fractional field $ \mathbb{K}^{G'}$.
Let $\mu \in G$ be $ \phi (\mathfrak{m})$, and define $\gamma : \mathbb{K}^{G'} \ra \mathbb{K}^{G'}$ to be the homomorphism induced by $\gamma (g)=\mu g\mu^{-1}$.
Then, we obtain the skew polynomial ring $\mathbb{K}^{G'}_{\gamma }[\tau^{\pm 1}] $, and a ring homomorphism
$$\nu: \Z[G] \lra \mathbb{K}_{\gamma}^{G'}[\tau^{\pm 1}] ; \ \ \ \ \ \sum_g n_g g \longmapsto \sum_g n_g g \mu^{-\phi (g)} \tau^{\phi (g)} . $$
Let $ \mathcal{A}$ be the division ring $\mathbb{K}^{G'} $, and $\kappa$ be $\gamma$.
Then, the following composite satisfies Assumption $(\dagger)$:
\begin{equation}\label{ooo222}\rho_G : \Z [ \pi_1(S^3 \setminus L)]\stackrel{\phi}{\lra} \Z[G] \stackrel{\nu}{\lra} \mathbb{K}^{G'}_{\gamma} [\tau^{\pm 1}] \hookrightarrow \mathbb{K}^{G'}_{\gamma}( \!( \tau) \!)
\end{equation}

Next, we review the higher order Alexander polynomial.
The ring $\mathbb{K}^{G'}_{\gamma} [\tau^{\pm 1}]$ is known to be a principal ideal domain since $\mathbb{K}^{G'}_{\gamma}$ is a skew field.
For a finitely generated right $\mathbb{K}^{G'}_{\gamma} [\tau^{\pm 1}]$-
module $H$, the elementary divisor theorem claims an isomorphism
$$ H \cong \bigoplus_{i : \ 1 \leq i \leq \ell} \mathbb{K}^{G'}_{\gamma} [\tau^{\pm 1}]/
p_i(t) \mathbb{K}^{G'}_{\gamma} [\tau^{\pm 1}] $$
for some $\ell \in \N$ and $p_i(t) \in \mathbb{K}^{G'}_{\gamma} [\tau^{\pm 1}]$ for $i \in \{ 1,\dots, \ell\}$.
Following \cite{C,Fri2}, we define ord$(H)$ by the product $p_1(t) \cdots p_{\ell}(t)$.
Although it is a subject to discuss 
which set this ord$(H)$ should be contained in,
according to \cite{Fri2}, we regard ord$(H)$ as an element in $ \mathbb{K}^{G'}_{\gamma} (\tau)^{\times}_{\rm ab} \cup \{0\} $ up to multiplication by an element of the form $k \tau ^e$ with $ k \in \mathbb{K}^{G'}_{\gamma}$ and $e \in \Z$.
As is known, ord$(H)$ depends on only $H$; see \cite[Theorem 3.1]{Fri2}.
Then, the order of the $i$-th homology
$$ \Delta_{i}^{\psi} := \mathrm{ord}H_i( S^3 \setminus L ; \mathbb{K}^{G'}_{\gamma}[\tau^{\pm 1}] ) \in \mathbb{K}^{G'}_{\gamma} (\tau)^{\times}_{\rm ab}/\{k \tau ^e \}_{k \in \mathbb{K}^{G'}_{\gamma}, e \in \Z} $$
is called {\it the higher order Alexander polynomial}.

To state Proposition \ref{prop188},
let us consider the natural inclusion $\lambda $ from the fractional field $ \mathbb{K}^{G'}_{\gamma} (\tau) $ into $ \mathbb{K}^{G'}_{\gamma}( \!( \tau) \!) $, since $\mathbb{K}^{G'}_{\gamma}( \!( \tau) \!) $ is also a field.
Then, by the isomorphism $ K_1(\mathbb{K}^{G'}_{\gamma}( \!( \tau) \!)^{\times }) \cong (\mathbb{K}^{G'}_{\gamma}( \!( \tau) \!)^{\times })_{\rm ab}$, the inclusion
yields the homomorphism
$$ \lambda_* : \mathbb{K}^{G'}_{\gamma} (\tau)^{\times}_{\rm ab}/\{k \tau ^e \}_{k \in \mathbb{K}^{G'}_{\gamma}, e \in \Z} \lra \mathcal{Q}_{\mathcal{A}, \gamma} / \{ \tau^e \}_{e\in \Z} .$$
Then, the pushforward of $\Delta_{1}^{\psi}$ turns out to be almost
the $K_1$-class $ \Delta_{\rho_G}^{K_1}$; More precisely,
\begin{prop}\label{prop188}
Suppose that the matrix $A_{F,W}$ is invertible. Then, $ \Delta_{i}^{\psi}$ is zero if $i > 1$, and
$ \Delta_{0}^{\psi} =1- \tau$.
Furthermore, the following equality holds: 
$$\lambda_* (\Delta_{1}^{\psi}) = \Delta_{\rho_G}^{K_1} \in \mathcal{Q}_{\mathcal{A}, \gamma} / \{ \tau^e \}_{e\in \Z} .$$
\end{prop}
\begin{proof}
From the cellular complex \eqref{ooo24}, we see that
the $i$-th homology is vanishes for $i>1$, i.e., $ \Delta_{i}^{\psi}=0$,
amd if $i=0$, we find
$$H_0( S^3 \setminus L ; \mathbb{K}^{G'}_{\gamma}[\tau^{\pm 1}] )\cong \mathbb{K}^{G'}_{\gamma}[\tau^{\pm 1}]/ (1- \tau) \mathbb{K}^{G'}_{\gamma}[\tau^{\pm 1}]. $$
Thus, $ \Delta_{0}^{\psi} =1- \tau$.

Let us consider the case $i=1$. Then, the main theorem of \cite{Fri2},
$$ \mathcal{T}( S^3 \setminus L,\rho_G)= \Delta_{1}^{\psi}/ \Delta_{0}^{\psi} \in \mathbb{K}^{G'}_{\gamma} (\tau)^{\times}_{\rm ab}/\{k \tau ^e \}_{k \in \mathbb{K}^{G'}_{\gamma}, e \in \Z}, $$
where $\rho_G$ is the composite in \eqref{ooo222}.
In the proof of Theorem \ref{cor11}, $\lambda_*(\mathcal{T}( S^3 \setminus L,\rho_G)) $ was shown to be $[A_{F,W}] \cdot (1 -\tau)^{-1}$. Hence, we have the required equality.
\end{proof}

However, the point is that, in general, the fractional field $\mathbb{K}^{G'}$ is quite incomprehensible from quantitative viewpoint; in particular,
it is difficult to compute the abelianization $\mathbb{K}^{G'}_{\gamma} (\tau)^{\times}_{\rm ab}$
and to distinguish whether two elements in $\mathbb{K}^{G'}_{\gamma} (\tau)^{\times}_{\rm ab}$ are equal or not.
By the reasons, it is believed that there is no example of computing their Alexander polynomial with non-triviality even if $L$ is an easy knot.

In contrast, in the definition of the $K_1$-class $ \Delta_{\rho}^{K_1} $,
we considered no fraction.
In \cite[\S\S 4--6]{Nos}, the author showed some non-trivial examples of $ \Delta_{\rho}^{K_1} $
without using fraction, but using a logarithm and the homomorphism $\Upsilon_*$.
This is a reason why we employ Novikov rings instead of fractional fields and polynomial ring.


\normalsize
DEPARTMENT OF
MATHEMATICS
TOKYO
INSTITUTE OF
TECHNOLOGY
2-12-1
OOKAYAMA
, MEGURO-KU TOKYO
152-8551 JAPAN

\end{document}